\documentclass[12pt]{article}
\usepackage{amssymb}
\usepackage{amsthm}
\usepackage{amsmath}
\usepackage{graphicx}
\usepackage{enumerate}
\newcommand{\xdownarrow}[1]{%
  {\left\downarrow\vbox to #1{}\right.\kern-\nulldelimiterspace}
}

\DeclareMathOperator{\Con}{Con}

\DeclareMathOperator{\DM}{DM}
\DeclareMathOperator{\BDM}{\mathbf{DM}}
\newtheorem{theorem}{Theorem}[section]
\newtheorem{definition}[theorem]{Definition}
\newtheorem{lemma}[theorem]{Lemma}

\newtheorem{remark}[theorem]{Remark}
\newtheorem{example}[theorem]{Example}
\newtheorem{corollary}[theorem]{Corollary}
\title{Sectionally pseudocomplemented posets}
\author{Ivan~Chajda, Helmut~L\"anger and Jan~Paseka}
\date{}
\begin{document}
\footnotetext[1]{Support of the research by \"OAD, project CZ~02/2019, and support of the research of the first author by IGA, project P\v rF~2019~015, is gratefully acknowledged.}
\maketitle
\begin{abstract}
The concept of a sectionally pseudocomplemented lattice was introduced in \cite C as an extension of relative pseudocomplementation for not necessarily distributive lattices. The typical example of such a lattice is the non-modular lattice N$_5$. The aim of this paper is to extend the concept of sectional pseudocomplementation from lattices to posets. At first we show that the class of sectionally pseudocompelemented lattices forms a variety of lattices which can be described by two simple identities. This variety has nice congruence properties. We summarize properties of sectionally pseudocomplemented posets and show differences to relative pseudocomplementation. We prove that every sectionally pseudocomplemented poset is completely $L$-semidistributive. We introduce the concept of congruence on these posets and show when the quotient structure becomes a poset again. Finally, we study the Dedekind-MacNeille completion of sectionally pseudocomplemented posets. We show that contrary to the case of relatively pseudocomplemented posets, this completion need not be sectionally pseudocomplemented but we present the construction of a so-called generalized ordinal sum which enables us to construct the Dedekind-MacNeille completion provided the completions of the summands are known. 
\end{abstract}
 
{\bf AMS Subject Classification:} 06A11, 06D15, 06B23

{\bf Keywords:} Sectional pseudocomplementation, poset, congruence, Dedekind-MacNeille completion, generalized ordinal sum.

\section{Introduction}

The concept of relative pseudocomplemented lattices was introduced by R.~P.~Dilworth in \cite D. The usefulness of this concept was shown in numerous papers and books, see e.g. the famous paper \cite{Ba} by R.~Balbes and the monograph \cite{Bi} by G.~Birkhoff. This concept was extended to posets recently by the first and second author in \cite{CLa}. Relatively pseudocomplemented lattices turn out to be distributive, a property which also holds for relatively pseudocomplemented posets (see \cite{CLa}). In order to extend relative pseudocomplementation in lattices to the non-distributive case, sectional pseudocomplementation was introduced in \cite C. The aim of the present paper is to extend sectional pseudocomplementation to posets which, of course, need not be distributive.

The concept of a sectionally pseudocomplemented lattice was introduced by the first author in \cite C. Recall that a lattice $(L,\vee,\wedge)$ is {\em sectionally pseudocomplemented} if for all $a,b\in L$ there exists the pseudocomplement of $a\vee b$ with respect to $b$ in $[b,1]$, in other words, there exists a greatest element $c$ of $L$ satisfying $(a\vee b)\wedge c=b$. In this case $c$ is called the {\em sectional pseudocomplement of $a$ with respect to $b$} and it will be denoted by $a*b$.

The aim of this paper is to extend this concept to posets.

\section{Properties of sectionally pseudocomplemented po\-sets and lattices}

Let $(P,\leq)$ be a poset, $a,b\in P$ and $A,B\subseteq P$. Recall that
\begin{align*}
L(A) & :=\{x\in P\mid x\leq y\text{ for all }y\in A\}, \\
U(A) & :=\{x\in P\mid y\leq x\text{ for all }y\in A\}.
\end{align*}
Instead of $L(\{a\})$, $L(\{a,b\})$, $L(A\cup\{a\})$, $L(A\cup B)$, $L(U(A))$ we simply write $L(a)$, $L(a,b)$, $L(A,a)$, $L(A,B)$, $LU(A)$, respectively. Analogously we proceed in similar cases. 
We also put $\xdownarrow{0.7586em} (A)=\{x\in P\mid x\leq y\text{ for some }y\in A\}$. 

We start with the following definition.

\begin{definition}
A {\em poset} $\mathbf P=(P,\leq)$ is called {\em sectionally pseudocomplemented} if for all $a,b\in P$ there exists a greatest $c\in P$ satisfying $L(U(a,b),c)=L(b)$. This element $c$ is called the {\em sectional pseudocomplement} $a*b$ of $a$ with respect to $b$. The poset $\mathbf P$ is called {\em strongly sectionally pseudocomplemented} if it is sectionally pseudocomplemented, it has a greatest element $1$ and it satisfies the condition $x\leq(x*y)*y$ {\rm(}which, as we will se later, is equivalent to the identity $x*((x*y)*y)\approx1${\rm)}.
\end{definition}

The following example shows that there really exist sectionally pseudocomplemented posets which are not strongly sectionally pseudocomplemented. Hence, we cannot expect that every sectionally pseudocomplemented poset satisfies the condition $x\leq(x*y)*y$.

\begin{example}
The poset visualized in Fig.~1

\vspace*{-4mm}

\[
\setlength{\unitlength}{7mm}
\begin{picture}(6,9)
\put(3,2){\circle*{.3}}
\put(1,4){\circle*{.3}}
\put(3,4){\circle*{.3}}
\put(5,4){\circle*{.3}}
\put(1,6){\circle*{.3}}
\put(3,6){\circle*{.3}}
\put(5,6){\circle*{.3}}
\put(3,8){\circle*{.3}}
\put(3,2){\line(-1,1)2}
\put(3,2){\line(1,1)2}
\put(3,4){\line(-1,1)2}
\put(3,4){\line(1,1)2}
\put(3,6){\line(-1,-1)2}
\put(3,6){\line(1,-1)2}
\put(3,8){\line(-1,-1)2}
\put(3,8){\line(0,-1)2}
\put(3,8){\line(1,-1)2}
\put(1,4){\line(0,1)2}
\put(5,4){\line(0,1)2}
\put(2.875,1.25){$a$}
\put(.35,3.85){$b$}
\put(5.4,3.85){$d$}
\put(3.4,3.85){$c$}
\put(.35,5.85){$e$}
\put(5.4,5.85){$g$}
\put(3.4,5.85){$f$}
\put(2.85,8.4){$1$}
\put(2.2,.3){{\rm Fig.~1}}
\end{picture}
\]

\vspace*{-3mm}

is sectionally pseudocomplemented and $*$ has the operation table
\[
\begin{array}{c|cccccccc}
* & a & b & c & d & e & f & g & 1 \\
\hline
a & 1 & 1 & c & 1 & 1 & 1 & 1 & 1 \\
b & g & 1 & g & g & 1 & 1 & g & 1 \\
c & f & f & 1 & f & 1 & f & 1 & 1 \\
d & e & e & e & 1 & e & 1 & 1 & 1 \\
e & d & f & g & d & 1 & f & g & 1 \\
f & a & e & c & g & e & 1 & g & 1 \\
g & b & b & e & f & e & f & 1 & 1 \\
1 & a & b & c & d & e & f & g & 1
\end{array}
\]
but it is not strongly sectionally pseudocomplemented since $c\not\leq a=f*a=(c*a)*a$.
\end{example}

Recall from \cite{CLa} or \cite P that the {\em relative pseudocomplement of $a$ with respect to $b$} is the greatest $d\in P$ satisfying $L(a,d)\subseteq L(b)$.

We are going to show that every sectionally pseudocomplemented lattice with $1$ is strongly sectionally pseudocomplemented.

The following lemma was proved in \cite C. For the convenience of the reader we provide a short proof.

\begin{lemma}
Every sectionally pseudocomplemented lattice $\mathbf L=(L,\vee,\wedge,*)$ satisfies $x\vee y\leq (x*y)*y$.
\end{lemma}

\begin{proof}
Because of $(x\vee y)\wedge y=y$ we have $y\leq x*y$ and hence $((x*y)\vee y)\wedge(x\vee y)\approx(x\vee y)\wedge(x*y)=y$ whence $x\vee y\leq(x*y)*y$.
\end{proof}

In a lattice $(P,\vee,\wedge)$ the equation $L(U(a,b),c)=L(b)$ is equivalent to $(a\vee b)\wedge c=b$.

\begin{example}\label{ex2}
The poset visualized in Fig.~2

\vspace*{-4mm}

\[
\setlength{\unitlength}{7mm}
\begin{picture}(6,9)
\put(3,2){\circle*{.3}}
\put(1,4){\circle*{.3}}
\put(5,5){\circle*{.3}}
\put(1,6){\circle*{.3}}
\put(3,8){\circle*{.3}}
\put(3,2){\line(-1,1)2}
\put(3,2){\line(2,3)2}
\put(1,4){\line(0,1)2}
\put(3,8){\line(-1,-1)2}
\put(3,8){\line(2,-3)2}
\put(2.875,1.25){$0$}
\put(.35,3.85){$a$}
\put(5.4,4.85){$b$}
\put(.35,5.85){$c$}
\put(2.85,8.4){$1$}
\put(2.2,.3){{\rm Fig.~2}}
\end{picture}
\]

\vspace*{-3mm}

is a strongly sectionally pseudocomplemented lattice and $*$ has the operation table
\[
\begin{array}{c|ccccc}
* & 0 & a & b & c & 1 \\
\hline
0 & 1 & 1 & 1 & 1 & 1 \\
a & b & 1 & b & 1 & 1 \\
b & c & a & 1 & c & 1 \\
c & b & a & b & 1 & 1 \\
1 & 0 & a & b & c & 1
\end{array}
\]
but this poset is not relatively pseudocomplemented since the relative pseudocomplement of $c$ with respect to $a$ does not exist.
\end{example}

It was shown in \cite C that the class of sectionally pseudocomplemented lattices forms a variety. However, the defining identities given in \cite C are rather complicated. We present some simpler identities as follows.

\begin{theorem}
The class of sectionally pseudocomplemented lattices forms a variety which besides the lattice axioms is determined by the following identities:
\begin{enumerate}
\item[{\rm(i)}] $z\vee y\leq x*((x\vee y)\wedge(z\vee y))$,
\item[{\rm(ii)}] $(x\vee y)\wedge(x*y)\approx y$.
\end{enumerate}
\end{theorem}

\begin{proof}
Let $\mathbf L=(L,\vee,\wedge)$ be a lattice and $a,b,c\in L$. First assume $\mathbf L$ to be sectionally pseudocomplemented. If $d:=(a\vee b)\wedge(c\vee b)$ then
\begin{align*}
a\vee d & =a\vee((a\vee b)\wedge(c\vee b))=a\vee(b\vee((a\vee b)\wedge(c\vee b)))= \\
& =(a\vee b)\vee((a\vee b)\wedge(c\vee b))=a\vee b, \\
c\vee d & =c\vee((a\vee b)\wedge(c\vee b))=c\vee(b\vee((a\vee b)\wedge(c\vee b)))= \\
& =(c\vee b)\vee((a\vee b)\wedge(c\vee b))=c\vee b
\end{align*}
and hence
\[
d\leq(a\vee d)\wedge(c\vee d)=(a\vee b)\wedge(c\vee b)=d,
\]
i.e.\ $(a\vee d)\wedge(c\vee d)=d$ which shows
\[
c\vee b=c\vee d\leq a*d=a*((a\vee b)\wedge(c\vee b))
\]
proving (i). Identity (ii) follows from the definition of $*$. Conversely, assume $\mathbf L$ to satisfy (i) and (ii). Then $(a\vee b)\wedge(a*b)=b$ according to (ii). If $(a\vee b)\wedge c=b$ then $b\leq c$ and hence $(a\vee b)\wedge(c\vee b)=(a\vee b)\wedge c=b$ whence
\[
c=c\vee b\leq a*((a\vee b)\wedge(c\vee b))=a*b
\]
according to (i). This shows that $a*b$ is the sectional pseudocomplement of $a$ with respect to $b$.
\end{proof}

We can prove that this variety has very strong congruence properties. Recall that an {\em algebra} $\mathbf A$ is called {\em arithmetical} if $\Theta\circ\Phi=\Phi\circ\Theta$ for all $\Theta,\Phi\in\Con\mathbf A$ and if the congruence lattice of $\mathbf A$ is distributive. (Here and in the following $\Con\mathbf A$ denotes the set of all congruences on $\mathbf A$.) Moreover, recall that an {\em algebra} $\mathbf A$ with $1$ is called {\em weakly regular} (see e.g.\ \cite{CEL}) if for arbitrary $\Theta,\Phi\in\Con\mathbf A$ we have that $[1]\Theta=[1]\Phi$ implies $\Theta=\Phi$. A {\em variety} $\mathcal V$ (with $1$) is called {\em arithmetical} or {\em weakly regular} if every of its members has the respective property.

\begin{theorem}
The variety $\mathcal V$ of sectionally pseudocomplemented lattices with $1$ is arithmetical and weakly regular.
\end{theorem}

\begin{proof}
Since every member of $\mathcal V$ is a lattice, $\mathcal V$ is congruence distributive. Moreover, since for
\[
p(x,y,z):=((x*y)*z)\wedge((z*y)*x).
\]
we have
\begin{align*}
p(x,x,y) & \approx((x*x)*y)\wedge((y*x)*x)\approx(1*y)\wedge((y*x)*x)\approx y\wedge((y*x)*x)\approx y, \\
p(y,x,x) & \approx((y*x)*x)\wedge((x*x)*y)\approx((y*x)*x)\wedge(1*y)\approx((y*x)*x)\wedge y\approx y,
\end{align*}
$\mathcal V$ is congruence permutable. Finally, since for
\[
t_1(x,y):=x*y\text{ and }t_2(x,y):=y*x
\]
we have that $t_1(x,y)=t_2(x,y)=1$ is equivalent to $x=y$, $\mathcal V$ is weakly regular (cf.\ \cite{CEL}).
\end{proof}

Evidently, every relatively pseudocomplemented lattice $(L,\vee,\wedge)$ is sectionally pseudocomplemented since for $a,b\in L$ we have $a*b=(a\vee b)\circ b$ where $*$ and $\circ$ denote sectional and relative pseudocomplementation, respectively. (Observe that $(a\vee b)\wedge b=b$ and hence $(a\vee b)\wedge((a\vee b)\circ b)=b$.) The following poset is an example of a sectionally pseudocomplemented poset which is neither a lattice nor relatively pseudocomplemented.

\begin{example}\label{ex1}
The poset visualized in Fig.~3

\vspace*{-2mm}

\[
\setlength{\unitlength}{7mm}
\begin{picture}(6,11)
\put(3,2){\circle*{.3}}
\put(1,4){\circle*{.3}}
\put(5,4){\circle*{.3}}
\put(1,6){\circle*{.3}}
\put(1,8){\circle*{.3}}
\put(5,8){\circle*{.3}}
\put(3,10){\circle*{.3}}
\put(3,2){\line(-1,1)2}
\put(3,2){\line(1,1)2}
\put(1,8){\line(0,-1)4}
\put(1,8){\line(1,-1)4}
\put(5,8){\line(-2,-1)4}
\put(5,8){\line(0,-1)4}
\put(3,10){\line(-1,-1)2}
\put(3,10){\line(1,-1)2}
\put(2.875,1.25){$0$}
\put(.35,3.85){$a$}
\put(5.4,3.85){$b$}
\put(.35,5.85){$c$}
\put(.35,7.85){$d$}
\put(5.4,7.85){$e$}
\put(2.85,10.4){$1$}
\put(2.2,.3){{\rm Fig.~3}}
\end{picture}
\]

\vspace*{-3mm}

is strongly sectionally pseudocomplemented and the operation table of $*$ looks as follows:
\[
\begin{array}{c|ccccccc}
* & 0 & a & b & c & d & e & 1 \\
\hline
0 & 1 & 1 & 1 & 1 & 1 & 1 & 1 \\
a & b & 1 & b & 1 & 1 & 1 & 1 \\
b & c & a & 1 & c & 1 & 1 & 1 \\
c & b & a & b & 1 & 1 & 1 & 1 \\
d & 0 & a & b & c & 1 & e & 1 \\
e & 0 & a & b & c & d & 1 & 1 \\
1 & 0 & a & b & c & d & e & 1
\end{array}
\]
Evidently, this poset is not a lattice. However, it is also not relatively pseudocomplemented since the relative pseudocomplement of $c$ with respect to $a$ does not exist.
\end{example}

In the following we list several important properties of sectionally pseudocomplemented posets.

\begin{theorem}\label{th2}
Let $\mathbf P=(P,\leq,*,1)$ be a sectionally pseudocomplemented poset with $1$. Then the following hold:
\begin{enumerate}
\item[{\rm(i)}] $x\leq y$ if and only if $x*y=1$,
\item[{\rm(ii)}] $x*x\approx x*1\approx1$,
\item[{\rm(iii)}] $1*x\approx x$,
\item[{\rm(iv)}] $x*(y*x)\approx1$,
\item[{\rm(v)}] $x*((y*x)*x)\approx1$,
\item[{\rm(vi)}] if $x*y=1$ or $y*x=1$ then $x*((x*y)*y)=1$,
\item[{\rm(vii)}] if $x*y=1$ then $(y*z)*(x*z)=1$,
\item[{\rm(viii)}] $L(U(x,y),x*y)\approx L(y)$.
\end{enumerate}
\end{theorem}

\begin{proof}
Let $a,b,c\in P$.
\begin{enumerate}
\item[(i)] The following are equivalent:
\begin{align*}
          a & \leq b, \\
     U(a,b) & =U(b), \\
    LU(a,b) & =L(b), \\
L(U(a,b),1) & =L(b), \\
        a*b & =1.
\end{align*}
\item[(ii)] follows from (i).
\item[(iii)] The following are equivalent:
\begin{align*}
L(U(1,a),b) & =L(a), \\
       L(b) & =L(a), \\
          b & =a.
\end{align*}
\item[(iv)] We have $L(U(a,b),b)=L(b)$ implies $b\leq a*b$.
\item[(v)] Because of (iv) we have $L(U(b*a,a),a)=L(b*a,a)=L(a)$ which shows $a\leq(b*a)*a$.
\item[(vi)] If $a\leq b$ then $a\leq b=1*b=(a*b)*b$ according to (iii) and (i), and if $b\leq a$ then $L(U(a*b,b),a)=L(a*b,a)=L(U(a,b),a*b)=L(b)$ and hence $a\leq(a*b)*b$.
\item[(vii)] If $a\leq b$ then $L(c)\subseteq L(U(a,c),b*c)\subseteq L(U(b,c),b*c)=L(c)$, i.e.\ $L(U(a,c),b*c)=L(c)$ whence $b*c\leq a*c$.
\item[(viii)] follows from the definition of $*$.
\end{enumerate}
\end{proof}

\begin{remark}\label{rem1}
Assertion {\rm(vii)} of Theorem~\ref{th2} says that $*$ is antitone in the first variable, i.e.\ $x\leq y$ implies $y*z\leq x*z$. Contrary to the case of relatively pseudocomplemented posets, sectional pseudocomplementation is not monotone in the second variable. Namely, in Example~\ref{ex1} we have $0\leq a$, but $b*0=c\not\leq a=b*a$. However, $*$ need not be monotone in the second variable also in sectionally pseudocomplemented lattices as the following example shows.
\end{remark}

\begin{example}
The lattice visualized in Fig.~4

\vspace*{-4mm}

\[
\setlength{\unitlength}{7mm}
\begin{picture}(5,11)
\put(3,2){\circle*{.3}}
\put(3,4){\circle*{.3}}
\put(1,6){\circle*{.3}}
\put(3,6){\circle*{.3}}
\put(5,6){\circle*{.3}}
\put(3,8){\circle*{.3}}
\put(3,10){\circle*{.3}}
\put(3,2){\line(0,1)8}
\put(1,6){\line(1,-1)2}
\put(1,6){\line(1,1)2}
\put(5,6){\line(-1,-2)2}
\put(5,6){\line(-1,2)2}
\put(2.875,1.25){$0$}
\put(3.4,3.85){$a$}
\put(.35,5.85){$b$}
\put(3.4,5.85){$c$}
\put(5.4,5.85){$d$}
\put(2.35,7.85){$e$}
\put(2.85,10.4){$1$}
\put(2.2,.3){{\rm Fig.~4}}
\end{picture}
\]

\vspace*{-3mm}

is sectionally pseudocomplemented and $*$ has the operation table
\[
\begin{array}{c|ccccccc}
* & 0 & a & b & c & d & e & 1 \\
\hline
0 & 1 & 1 & 1 & 1 & 1 & 1 & 1 \\
a & d & 1 & 1 & 1 & d & 1 & 1 \\
b & d & c & 1 & c & d & 1 & 1 \\
c & d & b & b & 1 & d & 1 & 1 \\
d & e & a & b & c & 1 & e & 1 \\
e & d & a & b & c & d & 1 & 1 \\
1 & 0 & a & b & c & d & e & 1
\end{array}
\]
Here we have $0<a$ and $b*0=d\parallel c=b*a$.
\end{example}

For every algebra $(A,*,1)$ of type $(2,0)$ and every subset $B$ of $A$ put
\begin{align*}
L(B) & :=\{x\in A\mid x*y=1\text{ for all }y\in B\}, \\
U(B) & :=\{x\in A\mid y*x=1\text{ for all }y\in B\}
\end{align*}
We are going to show that sectionally pseudocomplemented posets can be defined as certain groupoids.

\begin{theorem}\label{th3}
An algebra $(A,*,1)$ of type $(2,0)$ can be organized into a sectionally pseudocomplemented poset with $1$ if and only if the following hold:
\begin{enumerate}
\item[{\rm(i)}] $x*x\approx1$
\item[{\rm(ii)}] $x*y=y*x=1\Rightarrow x=y$,
\item[{\rm(iii)}] $x*y=y*z=1\Rightarrow x*z=1$,
\item[{\rm(iv)}] $L(U(x,y),x*y)=L(y)$,
\item[{\rm(v)}] $L(U(x,y),z)=L(y)\Rightarrow z*(x*y)=1$.
\end{enumerate}
\end{theorem}

\begin{proof}
The necessity of the conditions is clear. Conversely, assume (i) -- (v) to hold. Define a binary relation $\leq$ on $A$ by $x\leq y$ if $x*y=1$ ($x,y\in A$). Now \\
(i) implies reflexivity of $\leq$, \\
(ii) implies antisymmetry of $\leq$, \\
(iii) implies transitivity of $\leq$, \\
(iv) and (v) imply that $x*y$ is the sectional pseudocomplement of $x$ with respect to $y$. \\
Hence $(A,\leq)$ is a sectionally pseudocomplemented poset with sectional pseudocomplementation $*$.
\end{proof}

Recall that a {\em lattice} $(L,\vee,\wedge)$ is called {\em completely meet-semidistributive} if the following holds:
\[
\text{If }\emptyset\neq M\subseteq L,a,b\in L\text{ and }x\wedge a=b\text{ for all }x\in M\text{ then }(\bigvee M)\wedge a=b.
\]
For posets, we modify this concept as follows.

\begin{definition}
A {\em poset} $(P,\leq)$ is called {\em completely $L$-semidistributive} if the following holds:
\[
\text{If }\emptyset\neq M\subseteq P,a,b\in P\text{ and }L(x,a)=L(b)\text{ for all }x\in M\text{ then }L(U(M),a)=L(b).
\]
\end{definition}

\begin{theorem}
Let $\mathbf P=(P,\leq,*)$ be a sectionally pseudocomplemented poset. Then $\mathbf P$ is completely $L$-semidistributive.
\end{theorem}

\begin{proof}
If $\emptyset\neq M\subseteq P,a,b\in P$ and $L(x,a)=L(b)$ for all $x\in M$ then $b\leq a$ and therefore $L(U(a,b),x)=L(a,x)=L(b)$ for all $x\in M$ and hence $x\leq a*b$ for all $x\in M$ whence $a*b\in U(M)$ which finally implies
\[
L(b)\subseteq L(U(M),a)\subseteq L(a*b,a)=L(U(a,b),a*b)=L(b),
\]
i.e.\ $L(U(M),a)=L(b)$.
\end{proof}

\section{Congruences in sectionally pseudocomplemented posets}

Theorem~\ref{th2} (i) shows that in a sectionally pseudocomplemented poset $(P,\leq,*,1)$ with $1$, $\leq$ is uniquely determined by $*$. Let $(P,\leq,*,1)$ be a sectionally pseudocomplemented poset with $1$ and $\Theta\in\Con(P,*)$. We are interested in the question when $(P/\Theta,\leq')$ is a poset where $\leq'$ is defined by $[x]\Theta\leq'[y]\Theta$ if $[x]\Theta*[y]\Theta=[1]\Theta$ ($x,y\in P$). We will see that this is the case if $\Theta$ is convex, i.e.\ every class of $\Theta$ is a convex subset of $(P,\leq)$.

First we show that if $(P,\leq,*,1)$ is a finite sectionally pseudocomplemented poset with $1$ then $(P,*)$ has convex congruences.

In the following lemma and theorem we frequently use Theorem~\ref{th2} (vi).

\begin{lemma}\label{lemACC}
Let $(P,\leq,*,1)$ be a sectionally pseudocomplemented poset with $1$ satisfying the Ascending Chain Condition, let $a,b\in P$ and $\Theta\in\Con(P,*)$ and assume $a<b<(b*a)*a$ and $a\mathrel{\Theta}(b*a)*a$. Then $a\mathrel{\Theta}b$. 
\end{lemma}

\begin{proof}
Assume $(a,b)\notin\Theta$. Put $a_1:=a$, $a_2:=b$ and $a_n:=(a_{n-1}*a_{n-2})*a_{n-2}$ for $n\geq3$. Then $a_1<a_2<a_3$ and $a_3\mathrel{\Theta}a_1$. Now
\[
a_4=(a_3*a_2)*a_2\mathrel{\Theta}(a_1*a_2)*a_2=1*a_2=a_2.
\]
Moreover, $a_3\leq a_4$. Now $a_3=a_4$ would imply $a_1\mathrel{\Theta}a_3=a_4\mathrel{\Theta}a_2$, a contradiction. This shows $a_3<a_4$. Now
\[
a_5=(a_4*a_3)*a_3\mathrel{\Theta}(a_2*a_3)*a_3=1*a_3=a_3.
\]
Moreover, $a_4\leq a_5$. Now $a_4=a_5$ would imply $a_1\mathrel{\Theta}a_3\mathrel{\Theta}a_5=a_4\mathrel{\Theta}a_2$, a contradiction. This shows $a_4<a_5$. Going on in this way we would obtain an infinite strictly ascending chain $a_1<a_2<a_3<a_4<a_5<\cdots$ contradicting the Ascending Chain Condition. This shows $a\mathrel{\Theta}b$.
\end{proof}

Hence, we conclude

\begin{theorem}\label{theACC}
Let $(P,\leq,*,1)$ be a sectionally pseudocomplemented poset with $1$ satisfying the Ascending Chain Condition and let $\Theta\in\Con(P,*)$. Then $\Theta$ is convex.
\end{theorem}

\begin{proof}
Assume $a,b,c\in P$, $a<b<c$ and $(a,c)\in\Theta$. Then $b\leq(b*a)*a$. If $b=(b*a)*a$ then $b=(b*a)*a\mathrel{\Theta}(b*c)*a=1*a=a$. If $b<(b*a)*a$ then $a\mathrel{\Theta}b$ according to Lemma~\ref{lemACC}.
\end{proof}

\begin{corollary}\label{corACC}
If $(P,\leq,*,1)$ is a finite sectionally pseudocomplemented poset with $1$ then $(P,*)$ has convex congruences.
\end{corollary}

For the infinite case we have the following result.

\begin{theorem}
Let $(P,\leq,*,1)$ be a sectionally pseudocomplemented poset with $1$ such that $x,y\in P$, $x<y<1$, $x\not\prec y$ and $x<y*x$ together imply $\Theta(x,y)=P^2$. Then $(P,*)$ has convex congruences.
\end{theorem}

\begin{proof}
Let $\Theta\in\Con(P,*)$ and $a,b,c\in P$ and assume $a<b<c$ and $(a,c)\in\Theta$. If $c=1$ then
\[
a\mathrel{\Theta}1=a*b\mathrel{\Theta}1*b=b.
\]
If $c<1$ and $a<c*a$ then $\Theta(a,c)=P^2$ and hence $\Theta=P^2$ which implies $a\mathrel{\Theta}b$. If $c<1$ and $a=c*a$ then
\[
a=1*a=(a*a)*a\mathrel{\Theta}(c*a)*a=a*a=1=a*b=(c*a)*b\mathrel{\Theta}(a*a)*b=1*b=b.
\]
This shows that $\Theta$ is convex.
\end{proof}

Let $(P,\leq,*,1)$ be a sectionally pseudocomplemented poset with $1$ and $\Theta\in\Con(P,*)$. We define a binary relation $\leq'$ on $P/\Theta$ by $[x]\Theta\leq'[y]\Theta$ if $[x]\Theta*[y]\Theta=[1]\Theta$ ($x,y\in P$). Now we can prove

\begin{theorem}\label{th4}
Let $(P,\leq,*,1)$ be a strongly sectionally pseudocomplemented poset and let $a,b\in P$ and $\Theta$ a convex congruence on $(A,*)$. Then the following hold:
\begin{enumerate}
\item[{\rm(i)}] If $[a]\Theta\leq'[b]\Theta$ then there exists some $d\in[b]\Theta$ with $a\leq d$,
\item[{\rm(ii)}] if $a\leq b$ then $[a]\Theta\leq'[b]\Theta$,
\item[{\rm(iii)}] $(P/\Theta,\leq')$ is a poset.
\end{enumerate}
\end{theorem}

\begin{proof}
\
\begin{enumerate}
\item[(i)] Put $d:=(a*b)*b$. Then $d=(a*b)*b\mathrel{\Theta}1*b=b$ and $a\leq(a*b)*b=d$.
\item[(ii)] If $a\leq b$ then $a*b=1$ according to Theorem~\ref{th2} and hence $a*b\mathrel{\Theta}1$, i.e.\ $[a]\Theta\leq'[b]\Theta$.
\item[(iii)] Obviously, $\leq'$ is reflexive. Now assume $[a]\Theta\leq'[b]\Theta\leq'[a]\Theta$. Then, by (i), there exists some $d\in[b]\Theta$ with $a\leq d$. Because of $[d]\Theta=[b]\Theta\leq'[a]\Theta$ there exists some $e\in[a]\Theta$ with $d\leq e$. Since $a\leq d\leq e$, $(a,e)\in\Theta$ and $\Theta$ is convex we conclude $(a,d)\in\Theta$ showing $[a]\Theta=[d]\Theta=[b]\Theta$ proving antisymmetry of $\leq'$. Finally, let $c\in P$ and assume $[a]\Theta\leq'[b]\Theta\leq'[c]\Theta$. Then, by (i) there exists some $f\in[b]\Theta$ with $a\leq f$ and because of $[f]\Theta=[b]\Theta\leq'[c]\Theta$ some $g\in[c]\Theta$ with $f\leq g$. Because of $a\leq f\leq g$ we have $a\leq g$ which implies $[a]\Theta\leq'[g]\Theta=[c]\Theta$ by (ii), proving transitivity of $\leq'$.
\end{enumerate} 
\end{proof}

\begin{lemma}\label{anycongre}
Let $\mathbf P=(P,\leq,*,1)$ be a strongly sectionally pseudocomplemented poset, $a\in P$ and $\Theta\in\Con(P,*)$. Then $[a]\Theta$ is up-directed. Hence, if $\mathbf P$ satisfies the Ascending Chain Condition, then $[a]\Theta$ has a greatest element.
\end{lemma}

\begin{proof}
If $b,c\in[a]\Theta$ then
\[
(b*c)*c\in[(c*c)*c]\Theta=[1*c]\Theta=[c]\Theta=[a]\Theta,
\]
$b\leq(b*c)*c$ since $\mathbf P$ is strongly sectionally pseudocomplemented, and $c\leq(b*c)*c$ according to Theorem~\ref{th2} (v).
\end{proof}

From Theorem~\ref{th4} we have: If $(P,\leq,*,1)$ is a strongly sectionally pseudocomplemented poset, $\Theta$ a convex congruence on $(P,*)$, $a$ the greatest element of $[a]\Theta$ and $b$ the greatest element of $[b]\Theta$ then $a\leq b$ if and only if $[a]\Theta\leq'[b]\Theta$.

Now we solve the problem for which $\Theta\in\Con(P,*)$ the quotient $P/\Theta$ is again sectionally pseudocomplemented. We can state a sufficient condition.

\begin{definition}\label{strongcon}
Let $\mathbf P=(P,\leq,*)$ be a  sectionally pseudocomplemented poset and $\Theta\in\Con(P,*)$. We say that $\Theta$ is {\em strong} if the following holds: If $a,b\in P$, $a$ is the greatest element of $[a]\Theta$ and $b$ the greatest element of $[b]\Theta$ then $a*b$ is the greatest element of $[a*b]\Theta$. In this case we define $[a]\Theta*'[b]\Theta:=[a*b]\Theta$. 
\end{definition}

It is easy to see that if $\Theta$ is strong, $a,b,c,d\in P$, $a\leq b$ and $c$ is the greatest element of $[a]\Theta$ and $d$ the greatest element of $[b]\Theta$ then $c\leq d$.

\begin{theorem}
Let $\mathbf P=(P,\leq,*,1)$ be a strongly sectionally pseudocomplemented poset and $\Theta\in\Con(P,*)$ a strong congruence. If $\mathbf P$ satisfies the Ascending Chain Condition then $(P/\Theta,\leq',*',[1]\Theta)$ is a strongly sectionally pseudocomplemented poset.
\end{theorem}

\begin{proof}
Since  $\mathbf P$ satisfies the Ascending Chain Condition we know from Theorems~\ref{theACC} and \ref{th4} that $\Theta$ is convex and $(P/\Theta,\leq')$ a poset. Moreover, from Lemma~\ref{anycongre} we have that any congruence class of $\Theta$ has a greatest element. Put
\[
Q:=\{x\in P\mid x\text{ is the greatest element of }[x]\Theta\}.
\]
Then $(Q,*,1)$ is a subalgebra of $(P,*,1)$. Assume $a,b\in Q$. Since $\mathbf P$ is a strongly sectionally pseudocomplemented poset and $\Theta$ is strong we have $a,b\leq(a*b)*b$ and $a*b,(a*b)*b\in Q$. This implies
\begin{align*}
L_{Q}(b) & \subseteq L_{Q}(U_{Q}(a,b),a*b)\subseteq L_{Q}(U_{Q}((a*b)*b),a*b)=L_{Q}((a*b)*b,a*b)= \\
         & =L((a*b)*b,a*b)\cap Q=L(U(a*b,b),(a*b)*b)\cap Q=L(b)\cap Q=L_{Q}(b).
\end{align*}
Note that the first inclusion follows from the fact that $b\leq a*b$ and $b\in L_{Q}U_{Q}(a,b)$. The second inclusion follows from the fact that $L_{Q}U_{Q}(a,b)\subseteq L_{Q}U_{Q}((a*b)*b)$. The first equality follows from $L_{Q}U_{Q}((a*b)*b)=L_{Q}((a*b)*b)$. The second equality follows from the definition of $L_Q$. Since $\mathbf P$ is sectionally pseudocomplemented we have the next equality. \\
Now, let $c\in Q$ be such that $L_{Q}(U_{Q}(a,b),c)=L_{Q}(b)$. We have $U_{Q}(a,b)=U(a,b)\cap Q\subseteq U(a,b)$. Hence $LU(a,b)\subseteq LU_{Q}(a,b)=\xdownarrow{0.7586em}(L_{Q}U_{Q}(a,b))$. The last equality follows from the fact that $L_{Q}U_{Q}(a,b)\subseteq LU_{Q}(a,b)$ which yields $\xdownarrow{0.7586em} (L_{Q}U_{Q}(a,b))\subseteq\xdownarrow{0.7586em}(LU_{Q}(a,b))=LU_{Q}(a,b)$ and from the fact that $y\in LU_{Q}(a,b)$ implies $y\leq x\in L_{Q}U_{Q}(a,b)$ where $x$ is the greatest element of $[y]\Theta$. \\
We obtain $L(b)\subseteq L(U(a,b),c)\subseteq\xdownarrow{0.7586em}(L_{Q}U_{Q}(a,b)\cap L_{Q}(c))=\xdownarrow{0.7586em}(L_{Q}(b))=L(b)$ since $b,c\in Q$. Since $\mathbf P$ is a sectionally pseudocomplemented poset we have $c\leq a*b$, i.e., $(Q,*,1)$ is sectionally pseudocomplemented. Since $(Q,*,1)$ is a subalgebra of $(P,*,1)$ we have that $(Q,*,1)$ is strongly sectionally pseudocomplemented. Moreover, $x\mapsto[x]\Theta$ an isomorphism from $(Q,\leq,*,1)$ to $(P/\Theta,\leq',*',[1]\Theta)$ and hence $(P/\Theta,\leq',*',[1]\Theta)$ is also a strongly sectionally pseudocomplemented poset.
\end{proof}

The following lemma shows that in a strongly sectionally pseudocomplemented poset all principal congruences are given by the congruences of the form $\Theta(c,1)$.

\begin{lemma}
Let $(P,\leq,*,1)$ be a strongly sectionally pseudocomplemented poset, assume that every congruence on $(P,*)$ is convex and let $a,b\in P$ with $a\leq b$ . Then $\Theta(b*a,1)=\Theta(a,b)$.
\end{lemma}

\begin{proof}
Since $(a,b)\in\Theta(a,b)$ yields $(b*a,1)=(b*a,b*b)\in\Theta(a,b)$, we have $\Theta(b*a,1)\subseteq\Theta(a,b)$. Conversely, $(b*a,1)\in\Theta(b*a,1)$ yields $((b*a)*a,a)=((b*a)*a,1*a)\in\Theta(b*a,1)$ which because of $a\leq b\leq(b*a)*a$ and the convexity of $\Theta(b*a,1)$ implies $(a,b)\in\Theta(b*a,1)$, i.e.\ $\Theta(a,b)\subseteq\Theta(b*a,1)$.
\end{proof}

\section{Completion of sectionally pseudocomplemented \\
po\-sets}

Now we consider the Dedekind-MacNeille completion of sectionally pseudocomplemented posets.

It was shown by Y.~S.~Pawar (\cite P) that the Dedekind-MacNeille completion $\BDM(\mathbf Q)$ of a relatively pseudocomplemented poset $\mathbf Q$ is relatively pseudocomplemented and that the relative pseudocomplementation in $\DM(\mathbf Q)$ extends the relative pseudocomplementation in $\mathbf Q$ if $\mathbf Q$ is canonically embedded into $\BDM(\mathbf Q)$. In contrast to this the Dedekind-MacNeille completion of a strongly sectionally pseudocomplemented poset $\mathbf P$ need not be sectionally pseudocomplemented, even if $\mathbf P$ is finite and has a greatest element.

\begin{example}
Though the poset visualized in Fig.~5

\vspace*{-4mm}

\[
\setlength{\unitlength}{7mm}
\begin{picture}(4,5)
\put(0,2){\circle*{.3}}
\put(2,2){\circle*{.3}}
\put(4,2){\circle*{.3}}
\put(2,4){\circle*{.3}}
\put(2,4){\line(-1,-1)2}
\put(2,4){\line(0,-1)2}
\put(2,4){\line(1,-1)2}
\put(-0.15,1.25){$a$}
\put(1.85,1.25){$b$}
\put(3.85,1.25){$c$}
\put(1.85,4.4){$1$}
\put(1.2,.3){{\rm Fig.~5}}
\end{picture}
\]

\vspace*{-3mm}

is strongly sectionally pseudocomplemented with
\[
\begin{array}{c|cccc}
* & a & b & c & 1 \\
\hline
a & 1 & b & c & 1 \\
b & a & 1 & c & 1 \\
c & a & b & 1 & 1 \\
1 & a & b & c & 1
\end{array}
\]
as the operation table for $*$, its Dedekind-MacNeille completion visualized in Fig.~6

\vspace*{-4mm}

\[
\setlength{\unitlength}{7mm}
\begin{picture}(6,7)
\put(3,2){\circle*{.3}}
\put(1,4){\circle*{.3}}
\put(3,4){\circle*{.3}}
\put(5,4){\circle*{.3}}
\put(3,6){\circle*{.3}}
\put(3,2){\line(-1,1)2}
\put(3,2){\line(0,1)4}
\put(3,2){\line(1,1)2}
\put(3,6){\line(-1,-1)2}
\put(3,6){\line(1,-1)2}
\put(2.85,1.25){$0$}
\put(.35,3.85){$a$}
\put(3.3,3.85){$b$}
\put(5.3,3.85){$c$}
\put(2.85,6.4){$1$}
\put(2.2,.3){{\rm Fig.~6}}
\end{picture}
\]

\vspace*{-3mm}

is not sectionally pseudocomplemented since $a*0$ does not exist.
\end{example}

For our next investigations, we introduce the following useful concepts.

\begin{definition}\label{genord}
Let $(I,\leq)$ be a chain with greatest element $\top$ and smallest element $\bot$. Let $\mathbf P_i=(P_i,\leq_i),i\in I,$ be a family of posets such that $\mathbf P_\top$ has a greatest element $1$ and such that the following hold:
\begin{enumerate}[{\rm(i)}]
\item If $i\in I$ then there are $a,b\in P_i$ with $a<b$,
\item if $i,j,k\in I$ and $i<k<j$ then $P_i\cap P_j=\emptyset$, 
\item if $i,j\in I$ and $i<j$ then $|P_i\cap P_j|\leq 1$,
\item if $i,j\in I$, $i<j$ and $P_i\cap P_j=\{a\}$ then $P_i=L_{P_i}(a)$ and $P_j=U_{P_j}(a)$.
%\item $x=(b,i)$ for some element $b$ and some $i\in I$ implies 
%$x\not \in \bigcup_{j\in I} P_j$. 
\end{enumerate}
Put $P=\bigcup\limits_{i\in I}P_i$. For $a,b\in P$, say $a\in P_i$ and $b\in P_j$ with $i,j\in I$, define
\[
a\leq b\text{ if and only if }a=b\text{ or }(i=j\text{ and }a\leq_i b)\text{ or }i<j.
\]
We call $\mathbf P=(P,\leq)$ the {\em generalized ordinal sum} of $\mathbf P_i,i\in I$.
\end{definition}

It is elementary that $\mathbf P$ is a poset with a greatest element $1$.

%\begin{comment}
%\begin{lemma}\label{genorpos}
%Let $(P,\leq)$ be a generalized  ordinal sum of  posets $(P_i,\leq_i)$, $i\in I$. Then 
%$(P,\leq)$ is a poset with a greatest element $1$. 
%\end{lemma}
%\begin{proof} Let $a,b,c\in P$. Then there are $i,j,k\in I$ such that $a\in P_i$, 
%$b\in P_j$ and $c\in P_k$. 
%Since $a\in P_i$ we have that $a\leq_i a$, i.e., $a\leq a$. Assume now that $a\leq b$ and $b\leq a$. 
%If $i=j$ we have $a\leq_i b$ and $b\leq_i a$, i.e., $a=b$. Let $i<j$. The only possibility to obtain 
%$b\leq a$ is that $b\in P_i$ and $a\leq_i  b\leq_i a$, i.e., $P_i\cap P_j=\{b\}=\{a\}$. Let us check that $\leq$ 
%is transitive. Assume that $a\leq b$ and $b\leq c$. Suppose first that $i=j$ and $j=k$. 
%Then $a\leq_i b$ and $b\leq_i c$, hence $a\leq_i c$, i.e., $a\leq c$. Now, 
%let ($i=j$ and $j<k$) or  ($i<j$ and $j=k$) or  ($i<j$ and $j<k$). We obtain that 
%$i<k$, i.e., $a\leq c$ as well. 
%It remains to prove that $a\leq 1$. If $i=\top$ then $a\leq_{\top} 1$, i.e., $a\leq 1$. Otherwise 
%we have $i<\top$, i.e.,  $a\leq 1$.
%\end{proof}
%\end{comment}

Now, we can state some sufficient conditions under which the Dedekind-MacNeille completion of a sectionally pseudocomplemented poset is sectionally pseudocomplemented. By \cite{schmidt} the {Dedekind-MacNeille completion} of a poset $\mathbf P$ is (up to isomorphism) any complete lattice $\mathbf Q$ into which $\mathbf P$ can be supremum-densely and infimum-densely embedded (i.e., for every element $x\in Q$ there exist $M,N\subseteq P$ such that $x=\bigvee\varphi(M)=\bigwedge\varphi(N)$, where $\varphi\colon P\to Q$ is the embedding). We usually identify $P$ with $\varphi(P)$. In this sense $\mathbf Q$ preserves all infima and suprema existing in $\mathbf P$.

Now we turn our attention to a notion of a DM-yoked family of a generalized ordinal sum. The importance of this concept is based on the fact that  under natural assumptions 
(which are e.g. satisfied for a finite index set $I$)
the Dedekind-MacNeille completion of a generalized ordinal sum will be 
isomorphic to a generalized ordinal sum  of the respective DM-yoked family.

\begin{definition}\label{yoked}
Let $\mathbf P=(P,\leq)$ be the generalized ordinal sum of $\mathbf P_i=(P_i,\leq_i),i\in I$. We say that a family $\mathbf Q_i=(Q_i,\leq_i),i\in I,$ of posets is a {\em DM-yoked family of \/ $\mathbf P$} if the following conditions are satisfied:
\begin{enumerate}[{\rm(y1)}]
\item $\mathbf P_i$ is a sub-poset of \/ $\mathbf Q_i$ such that $\mathbf Q_i$ is the Dedekind-MacNeille completion of \/ $\mathbf P_i$ for every $i\in I$, 
\item if $i,j,k\in I$ and $i<k<j$ then $Q_i\cap Q_j=\emptyset$,
\item if $i,j\in I$ and $i<j$ then $|Q_i\cap Q_j|\leq 1$,
\item if $i,j\in I$, $i\prec j$, $P_i\cap P_j=\emptyset$, $\mathbf P_i$ has a greatest element and $\mathbf P_j$ has a smallest element then $Q_i\cap Q_j=\emptyset$,
\item if $i,j\in I$, $i\prec j$, $P_i\cap P_j=\emptyset$, $\mathbf P_i$ does not have a greatest element and $\mathbf P_j$ has a smallest element $0_{P_j}$ then $0_{P_j}$ is the greatest element of $\mathbf Q_i$,
\item if $i,j\in I$, $i\prec j$, $P_i\cap P_j=\emptyset$, $\mathbf P_j$ does not have a smallest element and $\mathbf P_i$ has a greatest element $1_{P_i}$ then $1_{P_i}$ is the smallest element of $\mathbf Q_j$,
\item if $i,j\in I$, $i\prec j$, $P_i\cap P_j=\emptyset$, $\mathbf P_j$ does not have a smallest element and $\mathbf P_i$ does not have a greatest element then the greatest element $1_{Q_i}$ of $\mathbf Q_i$ is the smallest element $0_{Q_j}$ of $\mathbf  Q_j$,
\item if $i,j\in I$, $i<j$ and $Q_i\cap Q_j=\{a\}$ then $a$ is the greatest element of $\mathbf Q_i$ and the smallest element of $\mathbf Q_j$.
\end{enumerate}
\end{definition}

The question when there exists a DM-yoked family for a given poset $\mathbf P$ which is a generalized ordinal sum of posets $\mathbf P_i=(P_i,\leq_i),i\in I,$ is positively answered in the following series of lemmas under the natural assumption that $P_j\cap (\BDM(P_i)\times \{i\})=\emptyset$ for all $i,j \in I$.

We will first need the following definition.
% that for any set 
%$X$ we have $(X\times I)\cap P_i=\emptyset$ for all $i\in I$. This means that the elements 
%of the form $(y,j)$, $j\in I$ are forbidden in all $P_i$, $i\in I$ because we will use such elements 
%in the construction of Dedekind-MacNeille completions of posets $P_i$, $i\in I$.

\begin{definition}\label{related}
Let $\mathbf P=(P,\leq)$ be the generalized ordinal sum of $\mathbf P_i=(P_i,\leq_i),i\in I$. We say that a family $\mathbf R_i=(R_i,\leq_i),i\in I,$ of posets is a {\em DM-related family of \/ $\mathbf P$} if the following conditions are satisfied:
\begin{enumerate}[{\rm(r1)}]
\item $\mathbf P_i$ is a sub-poset of \/ $\mathbf R_i$ such that $\mathbf R_i$ is the Dedekind-MacNeille completion of \/ $\mathbf P_i$ for every $i\in I$,
\item if $i\in I$ and $x\in R_i$ then $x\in R_i\setminus P_i$ if and only if $x=(A,i)$ and $A$ is a non-principal cut of $\BDM(\mathbf P_i)$.
\end{enumerate}
\end{definition}

\begin{lemma}\label{exyreld}
Let $\mathbf P=(P,\leq)$ be the generalized ordinal sum of $\mathbf P_i=(P_i,\leq_i),i\in I,$ such that $P_i\cap (\BDM(P_i)\times \{i\})=\emptyset$. Then a  DM-related family $\mathbf R_i=(R_i,\leq_i),i\in I,$ of \/ $\mathbf P$ exists.
\end{lemma}

\begin{proof}
Let  $i\in I$. We put
\[
R_i:=P_i\cup (\{A\in \BDM(P_i)\mid A \ \text{is not a principal cut in}\ \BDM(P_i)\}\times\{i\}).
\]
We have $P_i\cap(\{A\in \BDM(P_i)\mid A \ \text{is not a principal cut in}\ \BDM(P_i)\}\times\{i\})=\emptyset$. Define a mapping $\kappa_i\colon\BDM(P_i)\to R_i$ as follows:
\[
\kappa(A):=\left\{
\begin{array}{cl}
  a   & \text{if}\ A=L(a)\ \text{for some}\ a\in A \\
(A,i) & \text{if}\ A\ \text{is not a principal cut in}\ \BDM(P_i)
\end{array}
\right.
\]
for all $A\in\BDM(P_i)$. Clearly, $\kappa_i$ is a bijection. Let $x,y\in R_i$. We define $x\leq_i y$ if and only if $\kappa_i^{-1}(x)\subseteq \kappa_i^{-1}(x)$. Then $(R_i,\leq_i),i\in I,$ is a poset containing $P_i$ isomorphic with the poset $\BDM(\mathbf P_i)$. Hence the family 
$\mathbf R_i=(R_i,\leq_i),i\in I,$ is DM-related.
\end{proof}

\begin{lemma}\label{exyoked}
Let $\mathbf P=(P,\leq)$ be the generalized ordinal sum of $\mathbf P_i=(P_i,\leq_i),i\in I,$ such that $P_j\cap(\BDM(P_i)\times\{i\})=\emptyset$ for all $i,j\in I$.
%%$x=(b,i)$ for some element $b$ and some $i\in I$ implies $x\not \in \bigcup_{j\in I} P_j$. 
Then a DM-yoked family $\mathbf Q_i=(Q_i,\leq_i),i\in I,$ of \/ $\mathbf P$ exists.
\end{lemma}

\begin{proof}
Let $\mathbf R_i=(R_i,\leq_i),i\in I,$ be the DM-related family of \/ $\mathbf P$ which exists by \\
Lemma~\ref{exyreld}. We will proceed in two steps.

Step 1: Let $i\in I$. Assume that there exists some $j\in I$ with $i\prec j$ and $P_i\cap P_j=\emptyset$. If $\mathbf P_i$ does not have a greatest element
%and $\mathbf P_j$ has a smallest element $0_j$ 
we put $S_i:=(R_i\setminus\{1_{R_i}\})\cup\{0_{R_j}\}$ such that $0_{R_j}$ will be the greatest element of $S_i$ and the order $\leq_i$ on $S_i$ restricted to $R_i\setminus\{1_{R_i}\}$ will be the restriction of the order on $\mathbf R_i$. Clearly $P_i\subseteq S_i$ and $\mathbf P_i$ is a sub-poset of $\mathbf S_i$. If $\mathbf P_i$ does have a greatest element, we put $\mathbf S_i:=\mathbf R_i$. If $P_i\cap P_j=\{a\}$ then $a=1_{P_i}=0_{P_j}$. Hence also $a=1_{R_i}=0_{R_j}$ and we put $\mathbf S_i:=\mathbf R_i$. If there is no $j\in I$ such that $i\prec j$, we put again $\mathbf S_i:=\mathbf R_i$.
 
Step 2: Let $j\in I$. Assume there exists some $i\in I$ with $i\prec j$ and $P_i\cap P_j=\emptyset$. If $\mathbf P_j$ does not have a smallest element, we put $Q_j:=(S_j\setminus\{0_{S_j}\})\cup\{1_{S_i}\}$ such that $1_{S_i}$ will be the smallest element of $Q_j$ and the order $\leq_j$ on $Q_j$ restricted to $S_j\setminus\{0_{S_j}\}$ will be the restriction of the order on $\mathbf S_j$. Clearly $P_j\subseteq Q_j$ and $\mathbf P_j$ is a sub-poset of $\mathbf Q_j$. Otherwise we always put $\mathbf Q_j:=\mathbf S_j$. 
 
Let us now check that $\mathbf Q_i=(Q_i,\leq_i),i\in I,$ is a DM-yoked family of \/ $\mathbf P$.
\begin{enumerate}[{\rm(y1):}]
\item This follows immediately from the definition of $\mathbf Q_i$.
\item Let $i,j,k\in I$ and assume $i<k<j$. We always have $P_i\cap P_j=\emptyset$ and hence also $R_i\cap R_j=\emptyset$. \\
Step 1: Assume that $a\in S_i\cap S_j$ for some $a$. Then $a\not\in R_i$ or $a\not\in R_j$. Suppose first that $a\not\in R_i$. Then there exists some $l\in I$ with $i\prec l\leq k<j$ and $P_i\cap P_l=\emptyset$ and hence $a=0_{R_l}<1_{R_l}$. Now either $a=0_{P_l}$ or $a=(b,l)=0_{R_l}\not\in P_l$ for some $b$. Since $a\in S_j$, it can be only of a form $0_{P_j}$ or $0_{P_m}$ for $j\prec m$. But both cases are not possible (in the first case we would obtain $0_{P_j}=1_{P_l}=0_{P_l}$, in the second $P_l\cap P_m\not=\emptyset$). Assume now that $a\not\in R_j$. Then there exists some $m\in I$ such that $j\prec m$ and $a=0_{R_m}$. Since $a\in S_i$, it can be only of a form $1_{P_i}$ or $0_{R_n}$ for $i\prec n$. Since $i\prec n\leq k<j\prec m$, this is again not possible. Hence $S_i\cap S_j=\emptyset$. \\
Step 2: Suppose $a\in Q_i\cap Q_j$ for some $a$. Then $a\not\in S_i$ or $a\not\in S_j$. Assume now $a\not\in S_i$. Then there exists some $l\in I$ with $l\prec i<k<j$ and $P_l\cap P_i=\emptyset$ and hence $a=1_{S_l}\not\in P_l$ (otherwise we would have $a\in P_j$ which is not possible or $a=0_{P_m}$ for some $m\in I$ with $j\prec m$ which is also not possible). We conclude that either $a=(b,l)$ for some element $b$ or $a=0_{R_i}\in S_i$, a contradiction in the last case. Hence $a=(b,l)=1_{S_l}$. Since $a\in Q_j$, we have $a\not\in P_j$, i.e., $a=(c,j)$ for some element $c$ (which is not possible) or $a=1_{S_n}$ for some $n\prec j$ with $l\prec i<k\leq n<j$ (which is not possible) or $a=0_{P_m}$ for some $m\in I$ with $j\prec m$ (which is also not possible). Suppose $a\not\in S_j$. Then there exists some $n\in I$ with $i<k\leq n\prec j$ and $a=1_{S_n}\not\in P_n$ (otherwise we would have $a\in P_i$ which is not possible since then $a=0_{P_n}$ or $a=1_{P_r}$ for some $r\in I$ with $r\prec i$ which is not possible since $P_r\cap P_n=\emptyset$ or $a=0_{P_q}$ for some $q\in I$ with $i\prec q\leq k\leq n$ in which case $0_{P_q}=0_{1_q}$, a contradiction). Hence $a=1_{S_n}=(c,n)$ for some element $c$. Since $a\in Q_i$, we have that either there exists some $r\in I$ with $r\prec i$ such that $a=1_{S_r}$, i.e., either $a=(d,r)$ or $a=(e,i)$ for some elements $d,r$, a contradiction to $r<i<n$, or there exists some $q\in I$ with $i\prec q\leq k\leq n$ such that $a=0_{S_q}$ in which case $q=k=n$ and $a=1_{S_n}=0_{S_n}$, a contradiction. Hence $Q_i\cap Q_j=\emptyset$. 
\item Let $i,j\in I$ with $i<j$. If $i\not\prec j$ we know from (y2) that $|Q_i\cap Q_j|=0$. Hence we may assume $i\prec j$. Let $a\in Q_i\cap Q_j$. It  is enough to show that $a=1_{Q_i}=0_{Q_j}$ (which will give us also (y8)). Assume first $P_i\cap P_j\not=\emptyset$. Then $1_{P_i}=0_{P_j}$, $S_i=R_i$ and $Q_j=S_j$. Hence either $Q_j=R_j$ or there exists some $k\in I$ such that $j\prec k$, $P_j\cap P_k=\emptyset$ and ${\mathbf P}_j$ does not have a greatest element. Similarly, either $Q_i=R_i$ or there exists some $h\in I$ such that $h\prec i$, $P_h\cap P_i=\emptyset$ and ${\mathbf P}_i$ does not have a smallest element. We distinguish the following four cases: \\
\noindent{} $Q_j=R_j$ and $Q_i=R_i$: Since $a\in Q_i\cap Q_j$ we have $a\in P_i\cap P_j$, i.e., $a=1_{P_i}=0_{P_j}$. We have $a=1_{Q_i}=0_{Q_j}$. \\
\noindent{} $Q_j=R_j$ and there exists some $h\in I$ such that $h\prec i\prec j$, $P_h\cap P_i=\emptyset$ and ${\mathbf P}_i$ does not have a smallest element: Then either $a\in S_i=R_i$ (which yields that $a=1_{P_i}=0_{P_j}$) or $a=1_{S_h}$. If $a=1_{S_h}$ we have either $a=1_{P_h}\in P_h\cap P_j=\emptyset$ or $a=0_{S_i}=0_{R_i}\in R_j$, i.e., $a\in P_i\cap P_j$, a contradiction. \\
\noindent{} $Q_i=R_i$ and there exists some $k\in I$ such that $j\prec k$, $P_j\cap P_k=\emptyset$ and ${\mathbf P}_j$ does not have a greatest element: Since $a\in Q_j=S_j$ we have either $a\in P_j\cap P_i$, i.e., $a=1_{P_i}=0_{P_j}$ or $a=0_{R_k}$. If $a=0_{R_k}$ then either  
$a=0_{P_k}\in P_k\cap P_i=\emptyset$ or $a=0_{R_k}\in R_k\setminus P_k$, $a=(b,k)\not\in R_i=Q_i$, a contradiction. \\
\noindent{} There exist $h,k\in I$ such that $h\prec i \prec j\prec k$, $P_h\cap P_i=\emptyset=P_j\cap P_k$, ${\mathbf P}_j$ does not have a greatest element and ${\mathbf P}_i$ does not have a smallest element: We have $a\in Q_i\cap S_j$. Hence ($a\in R_j$ or $a=0_{R_k}$) and ($a\in R_i$ or $a=1_{R_h}$). We have four cases. Three of them can be settled as above. So assume that ($a=0_{R_k}$) and ($a=1_{R_h}$). If $a=0_{P_k}$ or $a=1_{P_h}$, we have $a\in P_k\cap P_h=\emptyset$, a contradiction. Hence $a=0_{R_k}\in R_k\setminus P_k$ and $a=1_{R_h}\in R_h\setminus P_h$, a contradiction to $h<k$. 
      
Assume now $P_i\cap P_j=\emptyset$. Then $R_i\cap R_j=\emptyset$. Clearly $S_i\cap S_j\subseteq\{0_{R_j}\}$. Assume first $S_i\cap S_j=\emptyset$. Then $Q_i\cap Q_j\subseteq\{1_{S_i}\}=\{1_{R_i}\}$. If $Q_i\cap Q_j= \{1_{S_i}\}=\{1_{R_i}\}$ then $1_{Q_i}=1_{S_i}=0_{Q_j}$. Suppose now $S_i\cap S_j= \{0_{R_j}\}$. Then $0_{R_j}=0_{S_j}=1_{S_i}$. If ${\mathbf P}_j$ does have a smallest element then $S_j=Q_j$ and $Q_i=S_i$ or $Q_i=(S_i\setminus\{0_{S_i}\})\cup\{1_{S_h}\}$ for some $h\in I$ with $h\prec i\prec j$. In the first case $Q_i\cap Q_j=S_i\cap S_j=\{0_{R_j}\}=\{0_{P_j}\}=\{0_{Q_j}\}$. But $1_{S_i}=1_{Q_i}$. Hence $Q_i\cap Q_j=\{1_{Q_i}\}$ as well. In the second case $Q_i\cap Q_j=((S_i\setminus\{0_{S_i})\cup\{1_{S_h}\})\cap S_j\subseteq(S_i\cup S_h)\cap S_j\subseteq(S_i\cup R_h)\cap S_j\subseteq\{0_{R_j}\}$. Moreover, $1_{S_i}=1_{Q_i}\in Q_i$ and $1_{S_i}=0_{Q_j}\in Q_j$. Hence $Q_i\cap Q_j=\{0_{Q_j}\}=\{1_{Q_i}\}$.
      
Assume ${\mathbf P}_j$ does not have a smallest element. Then $Q_j=(S_j\setminus \{0_{S_j}\})\cup\{1_{S_i}\}$ and $Q_i=S_i$ or $Q_i=(S_i\setminus\{0_{S_i})\cup\{1_{S_h}\}$ for some $h\in I$ with $h\prec i\prec j$. In the first case $Q_i\cap Q_j=S_i\cap ((S_j\setminus\{0_{S_j}\})\cup\{1_{S_i}\})=\{0_{R_j}\}=\{1_{S_i}\}$. Moreover, $0_{Q_j}=1_{S_i}=1_{Q_i}$. In the second case $Q_i\cap Q_j=((S_i\setminus \{0_{S_i})\cup\{1_{S_h}\})\cap((S_j\setminus\{0_{S_j}\})\cup\{1_{S_i}\})=\{1_{S_i}\}$. As above, $1_{S_i}=1_{Q_i}\in Q_i$ and $1_{S_i}=0_{Q_j}\in Q_j$.
       
Summarizing, we obtain that always $Q_i\cap Q_j=\emptyset$ or $Q_i\cap Q_j=\{a\}$ in which case $a=1_{Q_i}=0_{Q_j}$.
         
\item Assume $i,j\in I$, $i \prec j$, $P_i\cap P_j=\emptyset$, $\mathbf P_i$ has a greatest element $1_{P_i}$ and $\mathbf P_j$ has a smallest element $0_{P_j}$. Then $R_i\cap R_j=\emptyset$, $S_i=R_i$ and $Q_j=S_j$. Assume $a\in Q_i\cap Q_j$. Then $a\in Q_i\cap S_j$. We have either $a=1_{S_h}$ for some $h\in I$ with $h\prec i$ or $a\in S_i=R_i$, and either $a\in R_j$ or $a=0_{R_k}$ for some $k\in I$ with $j\prec k$. We can assume
%%first that $a=1_{S_h}$ 
$a\in S_i=R_i$ or $a\in R_h\cup R_i$ for some $h\prec i<j$ and $a\in S_j\subseteq R_j\cup R_k$. Moreover, $R_h\cap(R_j\cup R_k)=\emptyset$ and $R_i\cap (R_j\cup R_k)=\emptyset$, a contradiction. We conclude $Q_i\cap Q_j=\emptyset$.
 
\item Suppose $i,j\in I$, $i\prec j$, $P_i\cap P_j=\emptyset$, $\mathbf P_i$ does not have a greatest element and $\mathbf P_j$ has a smallest element $0_{P_j}$. We have $R_i\cap R_j=\emptyset$, $Q_j=S_j$ and $S_i=\{0_{R_j}\}\cup(R_i\setminus\{1_{R_i}\})$. Clearly, $0_{R_j}\in Q_i$ and $0_{R_j}\in S_j=Q_j$. Hence $0_{P_j}=0_{R_j}\in Q_i\cap Q_j$ is the greatest element of $\mathbf Q_i$.
  
\item Assume $i,j\in I$, $i\prec j$, $P_i\cap P_j=\emptyset$, $\mathbf P_j$ does not have a smallest element and $\mathbf P_i$ has a greatest element $1_{P_i}$. We obtain $S_i=R_i$, $Q_j=(S_j\setminus\{0_{S_j}\})\cup\{1_{S_i}\}$ and $R_i\cap R_j=\emptyset$. We have $1_{P_i}=1_{R_i}=1_{S_i}\in Q_j$ and $1_{S_i}\in Q_i$. Hence $1_{P_i}$ is the smallest element of $\mathbf Q_j$.
  
\item Assume $i,j\in I$, $i\prec j$, $P_i\cap P_j=\emptyset$, $\mathbf P_j$ does not have a smallest element and $\mathbf P_i$ does not have a greatest element. Then $S_i=\{0_{R_j}\}\cup(R_i\setminus\{1_{R_i}\})$ and $Q_j=(S_j\setminus\{0_{S_j}\})\cup\{1_{S_i}\}$. Since $1_{S_i}=0_{R_j}=0_{Q_j}$ and $1_{S_i}=1_{Q_i}$, we obtain that the greatest element $1_{Q_i}$ of $\mathbf Q_i$ is the smallest element $0_{Q_j}$ of $\mathbf Q_j$.
\end{enumerate}
\end{proof}

Finally, we can state our results on Dedekind-MacNeille completion of posets which are the generalized ordered sum of their parts $(P_i,\leq_i),i\in I$.

\begin{theorem}\label{genordDM}
Let $\mathbf P=(P,\leq)$ be the generalized ordinal sum of $\mathbf P_i=(P_i,\leq_i),i\in I,$ 
%such that $x=(b,i)$ for some element $b$ and some $i\in I$ implies $x\not \in \bigcup_{j\in I} P_j$ 
and $\mathbf Q_i=(Q_i,\leq_i),i\in I,$ be a DM-yoked family of \/ $\mathbf P$. Then the generalized ordinal sum $\mathbf Q=(Q,\leq)$ of $\mathbf Q_i=(Q_i,\subseteq),i\in I,$ exists. If any non-empty subset of $I$ has a maximal element then $\BDM(\mathbf P)\cong\mathbf Q$. 
\end{theorem}

\begin{proof}
First let us check that the assumptions of Definition~\ref{genord} are satisfied. Clearly, $\mathbf Q_i,i\in I,$ is a family of posets such that $\mathbf Q_\top$ has a greatest element $1$ (we may identify $1$ with $1_{P_\top}$) and condition (i) is satisfied by (y1). Condition (ii) follows from (y2) and condition (iii) from (y3). From (y8) we obtain condition (iv). Hence the generalized ordinal sum $\mathbf Q=(Q,\leq)$ of $\mathbf Q_i=(Q_i,\subseteq),i\in I,$ exists. \\
Assume now that any non-empty subset of $I$ has a maximal element and $I$ has a smallest element $\bot$. Let us show $\mathbf Q\cong\BDM(\mathbf P)$. Note that we have, for every $i\in I$, order isomorphisms $\varphi_i\colon Q_i\to\DM(P_i)$ and $\psi_i\colon\DM(P_i)\to Q_i$ defined by $\varphi_i(a):=L_{P_i}U_{P_i}(\{x\in P_i\mid x\leq a\})$ and $\psi_i(B):=\bigvee_{Q_i}(B\cap P_i)$ for all $a\in Q_i$ and $B\in\DM(P_i)$, and $\varphi_i\circ\psi_i=\mbox{id}_{\DM(P_i)}$, $\psi_i\circ\varphi_i=\mbox{id}_{Q_i}$. 

We define mappings $\varphi\colon Q\to\DM(P)$ and $\psi\colon\DM(P)\to Q$ as follows:
\[
\begin{array}{ccc}
\varphi(a):=LU(\{x\in P\mid x\leq a\}) & \text{ and } & \psi(B):=\left\{
\begin{array}{ll}
\psi_\bot(\emptyset)     & \text{if }B=\emptyset, \\
\bigvee_{Q_j}(B\cap P_j) & \text{otherwise}
\end{array}
\right.
\end{array}
\]
where $j:=\max\limits_{B\cap P_m\neq\emptyset}m$ ($a\in Q$, $B\in\DM(P)$). Clearly, $\varphi$ and $\psi$ are well-defined and order-preserving. Recall also that $\varphi(a)=\varphi_k(a)\cup\bigcup\limits_{m<k}P_m$ where $k:=\max\limits_{a\in Q_m}m$. We have
\[
\psi(\varphi(a))=\psi(\varphi_k(a)\cup\bigcup_{m<k}P_m)=\bigvee_{Q_j}((\varphi_k(a)\cup\bigcup_{m<k}P_m)\cap P_k)=\psi_k(\varphi_k(a))=a
\]
for all $a\in Q$, here $k:=\max\limits_{a\in Q_m}m$. Let $B\in\DM(P)$. If $B=\emptyset$ and $B=LU(B)$ then $\varphi(\psi(\emptyset))=\varphi(\psi_\bot(\emptyset))=\varphi_\bot (0_{\DM(P_\bot)})$. Assume $x\in \varphi_\bot(0_{\DM(P_\bot)})$. Then $x$ is the smallest element of $\mathbf P_\bot$, i.e., $x\leq p$ for all $p\in P$, i.e., $x\in B$, a contradiction. Hence $\varphi(\psi(\emptyset))=\emptyset$. Suppose now that $B\neq\emptyset$ and put $j:=\max\limits_{B\cap P_m\neq\emptyset}m$. Then
\[
\varphi(\psi(B))=\varphi(\bigvee_{Q_j}(B\cap P_j))=\varphi_j(\bigvee_{Q_j}(B\cap P_j))\cup\bigcup_{m<j}P_m=(B\cap P_j)\cup\bigcup_{m<j}P_m=B. 
\]
\end{proof}

Now we show that the construction of a generalized ordinal sum preserves the property of sectional pseudocomplementation.

\begin{theorem}\label{genordDMpseud}
Let $\mathbf P=(P,\leq)$ be a generalized ordinal sum of $\mathbf P_i=(P_i,\leq_i),i\in I,$ assume that $(P_i,\leq,*_i)$ are sectionally pseudocomplemented for all $i\in I$, that $j\in I$, $L_j(P_j)=\emptyset$ implies $U_s(P_s)=\emptyset$ where $s:=\max\limits_{m<j}m$ and that any non-empty subset of $I$ has a maximal element. Then ${\mathbf P}$ is sectionally pseudocomplemented.
\end{theorem}

\begin{proof}
Let $i,j\in I$, $a\in P_i$ and $b\in P_j$ such that $i$ and $j$ are maximal with this property. We put
\[
a*b:=\left\{
\begin{array}{ll}
1     & \text{if }a\leq b, \\
a*_ib & \text{if }a\not\leq b\text{ and }i=j, \\
b     & \text{if }a\not\leq b\text{ and }i>j.
\end{array}
\right.
\]
We prove that $a*b$ is the sectional pseudocomplement of $a$ and $b$ in $\mathbf P$. \\
Case 1. $a\leq b$. \\
We have $L(U(a,b),1)=L(b,1)=L(b)$. \\
Case 2. $a\not\leq b$ and $i=j$. \\
We have
\begin{align*}
L(U(a,b),a*_ib) & =L(U(a,b)\cap P_i,a*_ib)=(L(U(a,b)\cap P_i,a*_ib)\cap P_i)\cup\bigcup_{m<i}P_m= \\
                & =L_i(U_i(a,b),a*_i b)\cup\bigcup_{m<i}P_m=L_i(b)\cup\bigcup_{m<i}P_m=L(b).
\end{align*}
Now assume $L(U(a,b),c)=L(b)$, $c\in P_k$, $k\in I$. Then $b\leq c$ which implies $k\geq i$. Now $k>i$ would imply
\[
a\in LU(a,b)=L(U(a,b),c)=L(b),
\]
a contradiction. Hence $k=i$ and
\[
L_i(U_i(a,b),c)=L(U(a,b),c)\cap P_i=L(b)\cap P_i=L_i(b)
\]
which implies $c\leq a*_ib$. \\
Case 3. $a\not\leq b$ and $i>j$. \\
We have $L(U(a,b),b)=L(a,b)=L(b)$. Now assume $L(U(a,b),c)=L(b)$, $c\in P_k$, $k\in I$. Then $b\leq c$ which implies $k\geq j$ and $L(b)=L(a,c)$. \\
Case 3a. $i>k$. \\
We have $c\leq b$ which yields $b=c$. \\
Case 3b. $i<k$. \\
We have $L(b)=L(a,c)=L(a)$, i.e., $a\leq b$, a contradiction. \\
Case 3c. $i=k$. \\
We obtain $k>j$. Since $L(b)=L(a,c)$, $b=1_j$ is the greatest element of $\mathbf P_j$. If $\mathbf P_i$ has a smallest element $0_i$ then necessarily $b=0_i\in P_i$, a contradiction to the assumption that $j$ is the maximal index from $I$ with $b\in P_j$. Hence $\mathbf P_i$ does not have a smallest element. Put $r:=\max\limits_{m<k}m$. Assume first that $j<r<k=i$. Then there exist by Definition~\ref{genord} (i) elements $x,y\in P_r$ with $b\leq x<y\leq a,c$. We conclude $y\in L(a,c)$, $y\not\leq b$, a contradiction. Suppose now that $j=r<k=i$. Since $L_i(P_i)=\emptyset$ and $b$ is the greatest element of $\mathbf P_j$ we obtain $\{b\}=U_r(P_r)=\emptyset$, a contradiction. Hence the only possible case is 3a which yields $b=c$ and finally that ${\mathbf P}$ is sectionally pseudocomplemented.
\end{proof}

Altogether, we can summarize our results as follows.

\begin{corollary}\label{fingenordDMpseud}
Let ${\mathbf P}=(P,\leq)$ be the generalized ordinal sum of $\mathbf P_i=(P_i,\leq_i),i\in I,$ such that $P_j\cap (\BDM(P_i)\times \{i\})=\emptyset$ for all $i,j\in I$. Let $(\DM(\mathbf P_i),\leq_i,*_i)$ be sectionally pseudocomplemented for all $i\in I$ and assume that any non-empty subset of $I$ has a maximal element and that $j\in I$ and $L_j(P_j)=\emptyset$ imply $U_s(P_s)=\emptyset$ where $s:=\max\limits_{m<j}m$. Then ${\BDM(\mathbf P)}$ is sectionally pseudocomplemented.
\end{corollary}

\begin{proof}
From Lemma~\ref{exyoked} we obtain that there exists a DM-yoked family $\mathbf Q_i=(Q_i,\leq_i),i\in I,$ of \/ $\mathbf P$ such that $\mathbf Q_i$ is isomorphic to a sectionally pseudocomplemented poset $\BDM(\mathbf P_i)$ for every $i\in I$. From Theorem~\ref{genordDM} we know that $\BDM(\mathbf P)$ is isomorphic to the generalized ordinal sum $\mathbf Q$ of the DM-yoked family $\mathbf Q_i=(Q_i,\leq_i),i\in I,$ of ${\mathbf P}$. Since every $\mathbf Q_i$ is sectionally pseudocomplemented we have from Theorem~\ref{genordDMpseud} that $\mathbf Q$ and hence also ${\BDM(\mathbf P)}$ are sectionally pseudocomplemented.
\end{proof}

The situation described in Theorem~\ref{genordDMpseud} and Corollary~\ref{fingenordDMpseud} can be illustrated by the following example. 

\begin{example}\label{yokedexam}
Consider the sectionally pseudocomplemented poset $\mathbf P$ from \\
Example~\ref{ex1}. It is evident that $\mathbf P$ is the generalized ordinal sum of the sectionally pseudocomplemented posets $\mathbf P_1=(P_1,\leq)=(\{0,a,b,c\},\leq)$ and $\mathbf P_2=(P_2,\leq)=(\{d,e,1\},\leq)$. Of course, $P_1\cap P_2=\emptyset$. Hence the conditions of Definition~\ref{genord} are satisfied. Then $\BDM(\mathbf P_1)$ is the lattice {\rm N}$_5$ and $\BDM(\mathbf P_2)$ the four-element Boolean algebra, i.e.\ both are sectionally pseudocomplemented lattices. The generalized ordinal sum of $\BDM(\mathbf P_1)$ and $\BDM(\mathbf P_2)$ is visualized in Fig.~7. According to Theorem~\ref{genordDMpseud} it is again sectionally pseudocomplemented.

The Dedekind-MacNeille completion of the poset visualized in Fig.~3 is visualized in Fig.~7

\vspace*{-2mm}

\[
\setlength{\unitlength}{7mm}
\begin{picture}(6,11)
\put(3,2){\circle*{.3}}
\put(1,4){\circle*{.3}}
\put(5,4){\circle*{.3}}
\put(1,6){\circle*{.3}}
\put(1,8){\circle*{.3}}
\put(5,8){\circle*{.3}}
\put(3,10){\circle*{.3}}
\put(2.33,6.67){\circle*{.3}}
\put(3,2){\line(-1,1)2}
\put(3,2){\line(1,1)2}
\put(1,6){\line(0,-1)2}
\put(1,8){\line(1,-1)4}
\put(5,8){\line(-2,-1)4}
\put(3,10){\line(-1,-1)2}
\put(3,10){\line(1,-1)2}
\put(2.875,1.25){$0$}
\put(.35,3.85){$a$}
\put(5.4,3.85){$b$}
\put(.35,5.85){$c$}
\put(2.8,6.45){$f$}
\put(.35,7.85){$d$}
\put(5.4,7.85){$e$}
\put(2.85,10.4){$1$}
\put(2.2,.3){{\rm Fig.~7}}
\end{picture}
\]

\vspace*{-3mm}

Here $L(x)$ is abbreviated by $x$ for $x\in\{0,a,b,c,d,e,1\}$ and $f$ is an abbreviation of $L(d,e)$. The operation table of $*$ in $\BDM(\mathbf P)$ looks as follows:
\[
\begin{array}{c|cccccccc}
* & 0 & a & b & c & d & e & f & 1 \\
\hline
0 & 1 & 1 & 1 & 1 & 1 & 1 & 1 & 1 \\
a & b & 1 & b & 1 & 1 & 1 & 1 & 1 \\
b & c & a & 1 & c & 1 & 1 & 1 & 1 \\
c & b & a & b & 1 & 1 & 1 & 1 & 1 \\
d & 0 & a & b & c & 1 & e & e & 1 \\
e & 0 & a & b & c & d & 1 & d & 1 \\
f & 0 & a & b & c & 1 & 1 & 1 & 1 \\
1 & 0 & a & b & c & d & e & f & 1
\end{array}
\]
According to Corollary~\ref{fingenordDMpseud}, $\BDM(\mathbf P)$ is just the generalized ordinal sum of $\BDM(\mathbf P_1)$ and $\BDM(\mathbf P_2)$.
\end{example}

Authors' addresses:

Ivan Chajda \\
Palack\'y University Olomouc \\
Faculty of Science \\
Department of Algebra and Geometry \\
17.\ listopadu 12 \\
771 46 Olomouc \\
Czech Republic \\
ivan.chajda@upol.cz

Helmut L\"anger \\
TU Wien \\
Faculty of Mathematics and Geoinformation \\
Institute of Discrete Mathematics and Geometry \\
Wiedner Hauptstra\ss e 8-10 \\
1040 Vienna \\
Austria, and \\
Palack\'y University Olomouc \\
Faculty of Science \\
Department of Algebra and Geometry \\
17.\ listopadu 12 \\
771 46 Olomouc \\
Czech Republic \\
helmut.laenger@tuwien.ac.at

Jan Paseka \\
Masaryk University Brno \\
Faculty of Science \\
Department of Mathematics and Statistics \\
Kotl\'a\v rsk\'a 2 \\
611 37 Brno \\
Czech Republic \\
paseka@math.muni.cz
\end{document}